\documentclass[times,authoryear,12pt]{elsarticle}
\usepackage{natbib}
\date{\today}
\usepackage{geometry,txfonts}
\usepackage[labelfont=bf]{caption}
\usepackage{amsmath,amsfonts,amssymb,amsthm,booktabs,color,epsfig,graphicx,url}
\urldef\myurl\url{www.ualberta.ca/~dwiens/}
\usepackage{appendix}
\usepackage[utf8]{inputenc}
\usepackage[english]{babel}
\usepackage{fancyhdr}
\newtheorem{theorem}{Theorem}

\newtheorem{lemma}{Lemma}
\newtheorem{remark}{Remark}

 \makeatletter
\def\ps@pprintTitle{%
  \let\@oddhead\@empty
  \let\@evenhead\@empty
  \def\@oddfoot{\reset@font\hfil\thepage\hfil}
  \let\@evenfoot\@oddfoot
}
\makeatother

\def\func{}

\journal{U. Alberta preprint series}

\begin{document}
\bibliographystyle{natbib}
\begin{frontmatter}
\title{To ignore dependencies is perhaps not a sin}

\author[A1]{Douglas P. Wiens\corref{mycorrespondingauthor}}

\address[A1]{Mathematical \& Statistical Sciences,
	University of Alberta,
	Edmonton, Canada,  T6G 2G1
 \newline \newline \today}

\cortext[mycorrespondingauthor]{E-mail: \url{doug.wiens@ualberta.ca}. 
Supplementary material is at {\myurl}. 
\\A version of this work is in \textit{Biometrika} as DOI: 10.1093/biomet/asaf025 \bf{`On the minimax robustness against correlation and heteroscedasticity of ordinary least squares among generalized least squares estimates of regression'}.
}

\begin{abstract}
We present a result according to which certain
functions of covariance matrices are maximized at scalar multiples of the
identity matrix. In a statistical context in which such functions measure
loss, this says that the least favourable form of dependence is in fact
independence, so that a procedure optimal for i.i.d.\ data can be minimax.
In particular, the ordinary least squares (\textsc{ols}) estimate of a
correctly specified regression response is minimax among generalized least
squares (\textsc{gls}) estimates, when the maximum is taken over certain
classes of error covariance structures and the loss function possesses a
natural monotonicity property. An implication is that it can be not only
safe, but optimal to ignore such departures from the usual assumption of
i.i.d.\ errors. We then consider regression models in which the response
function is possibly misspecified, and show that \textsc{ols} is minimax if
the design is uniform on its support, but that this often fails otherwise.
We go on to investigate the interplay between minimax \textsc{gls}
procedures and minimax designs, leading us to extend, to robustness against
dependencies, an existing observation -- that robustness against model
misspecifications is increased by splitting replicates into clusters of
observations at nearby locations.
\end{abstract}

\end{frontmatter}

\section{Introduction and summary}

When carrying out a study, whether observational or designed, calling for a
regression analysis the investigator may be faced with questions regarding
possible correlations or heteroscedasticity within the data. If there are
such departures from the assumptions underlying the use of the ordinary
least squares (\textsc{ols}) estimates of the regression parameters, then
the use of generalized least squares (\textsc{gls}) might be called for. In
its pure form, as envisioned by \cite{a35}, this calls for the use of the
inverse of the covariance matrix $C$, i.e.\ the \textit{precision} matrix,
of the random errors. This is inconvenient, since $C$ is rarely known and,
even if there is some prior knowledge of its structure, before the study is
carried out there are no data from which accurate estimates of its elements
might be made. If a consistent estimate $\hat{C}^{-1}$ of the precision
matrix does exist, then one can employ `feasible generalized least squares'
estimation - see e.g. \cite{fjh84}. An example is the
Cochrane-Orcutt procedure (\cite{co49}), which can be applied
iteratively in AR(1) models. Otherwise a positive definite `pseudo
precision' matrix $P$ might be employed. With data $y$ and design matrix $X$
this leads to the estimate 
\begin{equation}
\hat{\theta}_{\text{\textsc{gls}}}=\arg \min_{\theta }\left\Vert
P^{1/2}\left( y-X\theta \right) \right\Vert ^{2}=\left( X^{\prime }PX\right)
^{-1}X^{\prime }Py.  \label{glse}
\end{equation}
In \cite{w24a} a similar problem was addressed, pertaining to designed
experiments whose data are to be analyzed by \textsc{ols}. A lemma, restated
below as Lemma 1, was used to show that certain commonly employed loss
functions, taking covariance matrices as their arguments and increasing with
respect to the Loewner ordering by positive semidefiniteness, are maximized
at scalar multiples of the identity matrix. This has the somewhat surprising
statistical interpretation that the least favourable form of dependence is
in fact independence. The lemma was used to show that the assumption of
uncorrelated and homoscedastic errors at the design stage of an experiment is
in fact a \textit{minimax} strategy, within broad classes of alternate
covariance structures.

In this article we study the implications of the lemma in the problem of
choosing between \textsc{ols} and \textsc{gls}. We first show that, when the
form of the regression response is accurately modelled, then it can be safe,
and indeed optimal -- in a minimax sense -- to ignore possible departures
from independence and homoscedasticity, varying over certain large classes
of such departures. This is because the common functions measuring the loss
incurred by \textsc{gls}, when the covariance matrix of the errors is $C$,
are \textit{maximized} when $C$ is a multiple of the identity matrix. But in
that case the best \textsc{gls} estimate is \textsc{ols}, i.e.\ \textsc{ols}
is a minimax procedure.

We then consider the case of misspecified regression models, in which bias
becomes a component of the integrated mean squared prediction error (\textsc{%
imspe}). The \textsc{imspe} is maximized over $C$ and over the departures
from the fitted linear model. We show that, if a \textsc{gls} with (pseudo)
precision matrix $P$ is employed, then the variance component of this
maximum continues to be minimized by $P=I$, i.e.\ by \textsc{ols}, but the
bias generally does not and, depending upon the design, \textsc{ols} can
fail to be a minimax procedure. We show however that if the design is
uniform on its support then \textsc{ols} is minimax. Otherwise, \textsc{ols}
can fail to be minimax when the design emphasizes bias reduction over
variance reduction to a sufficiently large extent.

We also construct minimax designs -- minimizing the maximum\ \textsc{imspe}
over the design -- and combine them with minimax choices of $P$. These
designs are often uniform on their supports, and so \textsc{ols} is a
minimax procedure in this context. The design uniformity is attained by
replacing the replicates that are a feature of `classically optimal' designs
minimizing variance alone by clusters of observations at nearby design
points.

A summary of our findings is that, if a sensible design is chosen, then 
\textsc{ols} is at least `almost' a minimax \textsc{gls} procedure, often exactly so. We
conclude that, for Loewner-increasing loss functions, and for covariance
matrices $C$ varying over the classes covered by Lemma 1, the simplicity of 
\textsc{ols} makes it a robust and attractive alternative to \textsc{gls}.

The computations for this article were carried out in \textsc{matlab;} the
code is available on the author's personal website.

\section{A useful lemma\label{sec: lemma}}

Suppose that $\left\Vert \cdot \right\Vert _{M}$ is a matrix norm, induced
by the vector norm $\left\Vert \cdot \right\Vert _{V}$, i.e.\ 
\begin{equation*}
\left\Vert C\right\Vert _{M}=\sup_{\left\Vert x\right\Vert _{V}\text{ }%
=1}\left\Vert Cx\right\Vert _{V}.
\end{equation*}%
We use the subscript `$M$' when referring to an arbitrary matrix norm, but
adopt special notation in the following cases:

\noindent (i) For the Euclidean norm $\left\Vert x\right\Vert
_{V}=(x^{\prime }x)^{1/2}$, the matrix norm is denoted $\left\Vert
C\right\Vert _{E}$ and is the spectral radius, i.e.\ the root of the maximum
eigenvalue of $C^{\prime }C$. This is the maximum eigenvalue of $C$ if $C$
is a covariance matrix, i.e.\ is symmetric and positive semidefinite.

\noindent (ii) For the sup norm $\left\Vert x\right\Vert
_{V}=\max_{i}\left\vert x_{i}\right\vert $, the matrix norm $\left\Vert
C\right\Vert _{\infty }$ is $\max_{i}\sum_{j}\left\vert c_{ij}\right\vert $,
the maximum absolute row sum.

\noindent (iii) For the 1-norm $\left\Vert x\right\Vert
_{V}=\sum_{i}\left\vert x_{i}\right\vert $, the matrix norm $\left\Vert
C\right\Vert _{1}$ is $\max_{j}\sum_{i}\left\vert c_{ij}\right\vert $, the
maximum absolute column sum. For symmetric matrices, $\left\Vert
C\right\Vert _{1}=\left\Vert C\right\Vert _{\infty }$.

Now suppose that the loss function in a statistical problem is $\mathcal{L}%
\left( C\right) $, where $C$ is an $n\times n$ covariance matrix and $%
\mathcal{L}\left( \mathbf{\cdot }\right) $ is non-decreasing in the Loewner
ordering: 
\begin{equation*}
A\preceq B\Rightarrow \mathcal{L}\left( A\right) \leq \mathcal{L}\left(
B\right) .
\end{equation*}%
Here $A\preceq B$ means that $B-A\succeq 0$, i.e.\ is positive semidefinite
(p.s.d.).

The following lemma is established in \cite{w24a}.

\begin{lemma}
For $\eta ^{2}>0$, covariance matrix $C$ and induced norm $\left\Vert
C\right\Vert _{M}$, define 
\begin{equation}
\mathcal{C}_{M}=\left\{ C\left\vert {}\right. C\succeq 0\text{ and }%
\left\Vert C\right\Vert _{M}\leq \eta ^{2}\right\} .  \label{eta-bound}
\end{equation}%
For the norm $\left\Vert \mathbf{\cdot }\right\Vert _{E}$ an equivalent
definition is%
\begin{equation*}
\mathcal{C}_{E}=\left\{ C\left\vert {}\right. 0\preceq C\preceq \eta
^{2}I_{n}\right\} .
\end{equation*}%
Then:

\noindent (i) In any such class $\mathcal{C}_{M}$, $\max_{\mathcal{C}_{M}}%
\mathcal{L}\left( C\right) =\mathcal{L}\left( \eta ^{2}I_{n}\right) $.

\noindent (ii) If $\mathcal{C}^{\prime }$ $\mathcal{\subseteq C}_{M}$ and $%
\eta ^{2}I_{n}\in \mathcal{C}^{\prime }$, then $\max_{\mathcal{C}^{\prime }}%
\mathcal{L}\left( C\right) =\mathcal{L}\left( \eta ^{2}I_{n}\right) $.
\end{lemma}

A consequence of (i) of this lemma is that if one is carrying out a
statistical procedure with loss function $\mathcal{L}\left( C\right) $, then
a version of the procedure which minimizes $\mathcal{L}\left( \eta
^{2}I_{n}\right) $ is \textit{minimax} as $C$ varies over $\mathcal{C}_{M}$.

 An interpretation of the lemma is that, in attempting to maximize loss by altering the correlations or increasing the variances of $C$, one should always choose the latter. But the procedures discussed in this article do not depend on the particular 
value of $\eta ^{2}$ -- its only role is to ensure that $\mathcal{C}_{M}$ is
large enough to contain the departures of interest.

\section{Generalized least squares regression estimates when the response is
correctly specified \label{sec: gls}}

Consider the linear model 
\begin{equation}
y=X\theta +\varepsilon  \label{lin model}
\end{equation}%
for $X_{n\times p}$ of rank $p$. Suppose that the random errors $\varepsilon 
$ have covariance matrix $C\in \mathcal{C}_{M}$. If $C$ is \textit{known}
then the `best linear unbiased estimate' is $\hat{\theta}_{\text{\textsc{blue%
}}}=$ $\left( X^{\prime }C^{-1}X\right) ^{-1}X^{\prime }C^{-1}y$. In the
more common case that the covariances are at best only vaguely known, an
attractive possibility is to use the generalized least squares estimate (\ref%
{glse}) for a given positive definite (pseudo) precision matrix $P$. If $%
P=C^{-1}$ then the \textsc{blue} is returned. A diagonal $P$ gives `weighted
least squares' (\textsc{wls}). Here we propose choosing $P$ according to the 
\textit{minimax} principle, i.e.\ to minimize the maximum value of an
appropriate function $\mathcal{L}\left( C\right) $ of the covariance matrix
of the estimate, as $C$ varies over $\mathcal{C}_{M}$.

For brevity we drop the `pseudo' and call $P$ a precision matrix. Since $%
\hat{\theta}_{\text{\textsc{gls}}}$ is invariant under multiplication of $P$
by a scalar, we assume throughout that%
\begin{equation}
tr\left( P\right) =n.  \label{constraint}
\end{equation}

The covariance matrix of $\hat{\theta}_{\text{\textsc{gls}}}$ is 
\begin{equation*}
\text{\textsc{cov}}\left( \hat{\theta}_{\text{\textsc{gls}}}\left\vert
{}\right. C,P\right) =\left( X^{\prime }PX\right) ^{-1}X^{\prime }PCPX\left(
X^{\prime }PX\right) ^{-1}.
\end{equation*}%
Viewed as a function of $C$ this is non-decreasing in the Loewner ordering,
so that if a function $\Phi $ is non-decreasing in this ordering, then 
\begin{equation*}
\mathcal{L}\left( C\left\vert {}\right. P\right) =\Phi \{\text{\textsc{cov}}(%
\hat{\theta}_{\text{\textsc{gls}}}\mid C,P)\}
\end{equation*}%
is also non-decreasing and the conclusions of the lemma hold:%
\begin{equation*}
\max_{\mathcal{C}_{M}}\mathcal{L}\left( C\left\vert {}\right. P\right) =%
\mathcal{L}\left( \eta ^{2}I_{n}\left\vert {}\right. P\right) =\Phi \left\{
\eta ^{2}\left( X^{\prime }PX\right) ^{-1}XP^{2}X\left( X^{\prime }PX\right)
^{-1}\right\} .
\end{equation*}%
But, by virtue of the Gauss-Markov Theorem and the monotonicity of  $\Phi $,  this last expression is minimized by $P=I_{n}$, i.e.\ by the \textsc{ols}
estimate $\hat{\theta}_{\text{\textsc{ols}}}=\left( X^{\prime }X\right)
^{-1}X^{\prime }y$, with minimum value 
\begin{equation} \label{minmax}
\max_{\mathcal{C}_{M}}\mathcal{L}\left( C\left\vert {}\right. I_{n}\right)
=\Phi \left\{ \eta ^{2}\left( X^{\prime }X\right) ^{-1}\right\} .
\end{equation}%

It is well-known that if $0\preceq \Sigma _{1}\preceq \Sigma _{2}$ then the $%
i$th largest eigenvalue $\lambda _{i}$ of $\Sigma _{2}$ dominates that of $%
\Sigma _{1}$, for all $i$. It follows that $\Phi $ is non-decreasing in the
Loewner ordering in the cases:

\noindent (i) $\Phi \left( \Sigma \right) =tr\left( \Sigma \right)
=\sum_{i}\lambda _{i}\left( \Sigma \right) $;

\noindent (ii) $\Phi \left( \Sigma \right) =det\left( \Sigma \right)
=\prod_{i}\lambda _{i}\left( \Sigma \right) $;

\noindent (iii) $\Phi \left( \Sigma \right) =\max_{i}\lambda _{i}\left(
\Sigma \right) $;

\noindent (iv) $\Phi \left( \Sigma \right) =tr\left( K\Sigma \right) $ for $%
K\succeq 0$.

\noindent Thus if loss is measured in any of these ways and $C\in \mathcal{C}%
_{M}$ then $\hat{\theta}_{\text{\textsc{ols}}}$ is minimax for $\mathcal{C}%
_{M}$ in the class of \textsc{gls} estimates.

Minimax procedures are sometimes criticized for dealing optimally with an
overly pessimistic least favourable case -- see \cite{h72} for a
discussion; such criticism does certainly not apply here.

\begin{remark}
It is of interest to compare (\ref{minmax}) with the maximum loss of the $\textsc{gls}$ estimate 
which assumes that the covariance is $C_{0} \neq I_{n}$ and takes $ P_{0} = C_{0}^{-1}$. 
The ratio of the two maximum losses, i.e. that of the $\textsc{gls}$ estimate to that of 
the $\textsc{ols}$ estimate, is \linebreak $\max_{\mathcal{C}_{M}}\mathcal{L}\left( C\left\vert {}\right. P_{0}\right)
 / \max_{\mathcal{C}_{M}}\mathcal{L}\left( C\left\vert {}\right. I_{n}\right)$, with each maximum 
attained at $\eta^{2} I_{n}$. This is of course $\ge 1$ but can be arbitrarily large. 
See the Appendix for a simple example, with $C_{0}$ as in Example 2 below and
 $\Phi \left( \Sigma \right) =\det\left( \Sigma \right)$, in which this ratio is unbounded. 
\end{remark}

In each of the following examples, we posit a particular covariance
structure for $C$, a norm $\left\Vert C\right\Vert _{M}$, a bound $\eta ^{2}$
and a class $\mathcal{C}^{\prime }$ for which $C\in \mathcal{C}^{\prime
}\subseteq \mathcal{C}_{M}$. In each case $\eta ^{2}I_{n}\in \mathcal{C}%
^{\prime }$, so that statement (ii) of the lemma applies and $\hat{\theta}_{%
\text{\textsc{ols}}}$ is minimax for $\mathcal{C}^{\prime }$ (and for all of 
$\mathcal{C}_{M}$ as well) and with respect to any of the criteria (i) --
(iv).\medskip

\noindent \textbf{Example 1: Independent, heteroscedastic errors}. Suppose
that $C=diag(\sigma _{1}^{2},...,\sigma _{n}^{2})$. Then the discussion
above applies if $\mathcal{C}^{\prime }$ is the subclass of diagonal members
of $\mathcal{C}_{E}$ for $\eta ^{2}=\max_{i}\sigma _{i}^{2}$.\medskip

\noindent \textbf{Example 2}: \textbf{Equicorrelated errors}. Suppose that
the researcher fears that the observations are possibly weakly correlated,
and so considers $C=\sigma ^{2}\left( \left( 1-\rho \right) I_{n}+\rho
1_{n}1_{n}^{\prime }\right) $, with $\left\vert \rho \right\vert \leq \rho
_{\max }$. If $\rho $ $\geq 0$ then $\left\Vert C\right\Vert _{1}=\left\Vert
C\right\Vert _{\infty }=\left\Vert C\right\Vert _{E}=~\sigma ^{2}\left\{
1+\left( n-1\right) \left\vert \rho \right\vert \right\} $, and we take $%
\eta ^{2}=\sigma ^{2}\left\{ 1+\left( n-1\right) \rho _{\max }\right\} $. If 
$\mathcal{C}^{\prime }$ is the subclass of $\mathcal{C}_{1}$ or $\mathcal{C}%
_{\infty }$ or $\mathcal{C}_{E}$ defined by the equicorrelation structure,
then minimaxity of $\hat{\theta}_{\text{\textsc{ols}}}$ for any of these
classes follows. If $\rho <0$ then this continues to hold for $\mathcal{C}%
_{1}=\mathcal{C}_{\infty }$, and for $\mathcal{C}_{E}$ if $\eta ^{2}=\sigma
^{2}\left( 1+\rho _{\max }\right) $.\medskip

\noindent \textbf{Example 3: }\textsc{ma}{\small (1)} \textbf{errors}.
Assume first that the random errors are homoscedastic but are possibly
serially correlated, following an \textsc{ma}{\small (1)} model with 
\textsc{corr}$(\varepsilon _{i},\varepsilon _{j})=\rho I\left( \left\vert
i-j\right\vert =1\right) $ and with $\left\vert \rho \right\vert \leq \rho
_{\max }$. Then $\left\Vert C\right\Vert _{1}=\left\Vert C\right\Vert
_{\infty }\leq \sigma ^{2}\left( 1+2\rho _{\max }\right) =\eta ^{2}$, and in
the discussion above we may take $\mathcal{C}^{\prime }$ to be the subclass
-- containing $\eta ^{2}I_{n}$ -- defined by $c_{ij}=0$ if $\left\vert
i-j\right\vert >1$. If the errors are instead heteroscedastic, then $\sigma
^{2}$ is replaced by $\max_{i}\sigma _{i}^{2}$.\medskip

\noindent \textbf{Example 4: }\textsc{ar}{\small (1)}\textbf{\ errors}. It
is known -- see for instance \cite{t99}, p.\ 182 -- that the eigenvalues
of an \textsc{ar}{\small (1)} autocorrelation matrix with autocorrelation
parameter $\rho $ are bounded, and that the maximum eigenvalue $\lambda
\left( \rho \right) $ has $\lambda ^{\ast }=$ $\max_{\rho }\lambda \left(
\rho \right) >\lambda \left( 0\right) =1$. Then, again under
homoscedasticity, the covariance matrix $C$ has $\left\Vert C\right\Vert
_{E}\leq \sigma ^{2}\lambda ^{\ast }=\eta ^{2}$, and the discussion above
applies when $\mathcal{C}^{\prime }$ is the subclass defined by the
autocorrelation structure.\medskip

\noindent \textbf{Example 5: All of the above. }If $\mathcal{C}$ is the
union of the classes of covariance structures employed in Examples 1-4 ,
then the maximum loss over $\mathcal{C}$ is attained at $\eta _{0}^{2}I_{n}$%
, where $\eta _{0}^{2}$ is the maximum of those in these four examples. Then 
$\hat{\theta}_{\text{\textsc{ols}}}$ is minimax robust against the union of
these classes, since $\eta _{0}^{2}I_{n}$ is in each of them.

\section{Minimax precision matrices in misspecified response models\label%
{sec: minimax weights}}

Working in finite design spaces $\chi =\left\{ x_{i}\right\}
_{i=1}^{N}\subset \mathbb{R}^{d}$, and with $p$-dimensional regressors $%
f\left( x\right) $, \cite{w18} studied design problems for possibly
misspecified regression models 
\begin{equation}
Y\left( x\right) =f^{\prime }\left( x\right) \theta +\psi \left( x\right)
+\varepsilon , \label{alternate response}
\end{equation}%
with the unknown contaminant $\psi $ ranging over a class $\Psi $ and
satisfying, for identifiability of $\theta $, the orthogonality condition 
\begin{equation}
\sum_{x\in \chi }f\left( x\right) \psi \left( x\right) =0_{p\times 1},
\label{orthogonality}
\end{equation}%
as well as a bound%
\begin{equation}
\sum_{x\in \chi }\psi ^{2}\left( x\right) \leq \tau ^{2}.  \label{tau-bound}
\end{equation}%
For designs $\xi $ placing mass $\xi _{i}$ on $x_{i}\in \chi $, he took $%
\hat{\theta}=\hat{\theta}_{\text{\textsc{ols}}}$, loss function \textsc{imspe%
}: 
\begin{equation*}
\mathcal{I}\left( \psi ,\xi \right) =\sum_{x\in \chi }E[f^{\prime }\left(
x\right) \hat{\theta}-E\{Y\left( x\right) \}]^{2},
\end{equation*}%
and found designs minimizing the maximum, over $\psi $, of $\mathcal{I}%
\left( \psi ,\xi \right) $.

In \cite{w18} the random errors $\varepsilon _{i}$ were assumed to be
i.i.d.; now suppose that they instead have covariance matrix $C\in \mathcal{C%
}_{M}$ and take $\hat{\theta}=\hat{\theta}_{\text{\textsc{gls}}}$ with
precision matrix $P$. Using (\ref{orthogonality}), and emphasizing the
dependence on $C$ and $P$, $\mathcal{I}\left( \psi ,\xi \right) $ decomposes
as 
\begin{equation}
\mathcal{I}\left( \psi ,\xi \left\vert {}\right. C,P\right) =\sum_{x\in \chi
}f^{\prime }\left( x\right) \text{\textsc{cov}}\left( \hat{\theta}\left\vert
{}\right. C,P\right)f\left( x\right) +\sum_{x\in \chi }f^{\prime }\left(
x\right) b_{\psi ,P}b_{\psi ,P}^{\prime }f\left( x\right) +\sum_{x\in \chi
}\psi ^{2}\left( x\right) .  \label{I-cond0}
\end{equation}%
Here $b_{\psi ,P}=E(\hat{\theta})-\theta $ is the bias. Denote by $\psi _{X}$
the $n\times 1$ vector consisting of the values of $\psi $ corresponding to
the rows of $X$, so that $b_{\psi ,P}=\left( X^{\prime }PX\right)
^{-1}X^{\prime }P\psi _{X}$.

To express these quantities in terms of the design, define a set of $n\times
N$ indicator matrices 
\begin{equation*}
\mathcal{J}=\left\{ J\in \left\{ 0,1\right\} ^{n\times N}\left\vert
{}\right. J^{\prime }J\text{ is diagonal with trace }n\right\} .
\end{equation*}%
There is a one-one correspondence between $\mathcal{J}$ and the set of $n$%
-point designs on $\chi $. Given $J$, with 
\begin{equation*}
J^{\prime }J\equiv D=diag\left( n_{1},...,n_{N}\right) ,
\end{equation*}%
the $j$th column of $J$ contains $n_{j}$ ones, specifying the number of
occurrences of $x_{j}$ in a design, which thus has design vector $\xi
=n^{-1}J^{\prime }1_{n}=\left( n_{1}/n,...,n_{N}/n\right) ^{\prime }$.
Conversely, a design determines $J$ by $J_{ij}=I\left( \text{the }i\text{th
row of }X\text{ is }f^{\prime }\left( x_{j}\right) \right) $. The rank $q$
of $D$ is the number of support points of the design, assumed $\geq p$.

Define $F_{N\times p}$ to be the matrix with rows $\left\{ f^{\prime }\left(
x_{i}\right) \right\} _{i=1}^{N}$. Then $X=JF$ and, correspondingly, $\psi
_{X}=J\bar{\psi}$ for $\bar{\psi}=\left( \psi \left( x_{1}\right) ,...,\psi
\left( x_{N}\right) \right) ^{\prime }$.

The proofs for this section are in the appendix.

\begin{theorem}
\label{thm: maxima}For $\eta $ as at (\ref{eta-bound}) and $\tau $ as at (%
\ref{tau-bound}), define $\nu =\tau ^{2}/\left( \tau ^{2}+\eta ^{2}\right) $%
; this is the relative importance to the investigator of errors due to bias
rather than to variation. Then for a design $\xi $ and precision matrix $P$,
the maximum of $\mathcal{I}\left( \psi ,\xi \left\vert {}\right. C,P\right) $
as $C$ varies over $\mathcal{C}_{M}$ and $\psi $ varies over $\Psi $ is $%
\left( \tau ^{2}+\eta ^{2}\right) \times $ 
\begin{equation}
\mathcal{I}_{\nu }\left( \xi ,P\right) =\left( 1-\nu \right) \mathcal{I}%
_{0}\left( \xi ,P\right) +\nu \mathcal{I}_{1}\left( \xi ,P\right) ,
\label{max loss}
\end{equation}%
where 
%TCIMACRO{\TeXButton{TeX field}{\begin{subequations}}}%
%BeginExpansion
\begin{subequations}%
%EndExpansion
\label{max losses}%
\begin{eqnarray}
\mathcal{I}_{0}\left( \xi ,P\right) &=&tr\left\{ \left( Q^{\prime }UQ\right)
^{-1}\left( Q^{\prime }VQ\right) \left( Q^{\prime }UQ\right) ^{-1}\right\} ,
\label{max var} \\
\mathcal{I}_{1}\left( \xi ,P\right) &=&ch_{\max }\left\{ \left( Q^{\prime
}UQ\right) ^{-1}Q^{\prime }U^{2}Q\left( Q^{\prime }UQ\right) ^{-1}\right\} ,
\label{max bias} \\
U_{N\times N} &=&J^{\prime }PJ\text{\ and }V_{N\times N}=J^{\prime }P^{2}J. 
\notag
\end{eqnarray}%
%TCIMACRO{\TeXButton{TeX field}{\end{subequations}}}%
%BeginExpansion
\end{subequations}%
%EndExpansion
Here the columns of $Q_{N\times p}$ form an orthogonal basis for the column
space $\func{col}\left( F\right) $, $J$ is the indicator matrix of the
design $\xi $, and $ch_{\max }$ denotes the maximum eigenvalue of a matrix.
\end{theorem}

\begin{remark}
We assume throughout that the design is such that $X^{\prime }PX\succ 0$,
implying that $Q^{\prime }UQ\succ 0$.
\end{remark}

\begin{remark}
An investigator might decide beforehand to use \textsc{ols}, and then design
to minimize $\mathcal{I}_{\nu }\left( \xi ,P\right) =\mathcal{I}_{\nu
}\left( \xi ,I_{n}\right) $. This is a well-studied problem, solved for
numerous response models under the assumption of i.i.d. errors -- see \cite{w15} 
for a review. By virtue of Theorem \ref{thm: maxima} these designs
enjoy the additional property of being minimax against departures $C\in 
\mathcal{C}_{M}$.
\end{remark}

\subsection{Simulations}

%TCIMACRO{\TeXButton{B}{\begin{table}[p]\centering}}%
%BeginExpansion
\begin{table}[tbp]\centering%
%EndExpansion
\begin{tabular}{ccccccccc}
\multicolumn{9}{c}{Table 1. \ Minimax precision matrices; multinomial
designs:} \\ 
\multicolumn{9}{c}{means of performance measures $\pm $ $1$ standard error.}
\\ \hline\hline
Response & $N$ & $\nu $ & $\%I_{n}$ & $\mathcal{I}_{\nu }\left( \xi
,I_{n}\right) $ & $\mathcal{I}_{\nu }\left( \xi ,P^{\nu }\right) $ & $%
T_{1}\left( \%\right) $ & $T_{2}\left( \%\right) $ & $T_{3}\left( \%\right) $
\\ \hline
& $11$ & $.5$ & $1$ & $3.34\pm .11$ & $3.19\pm .11$ & $4.16\pm .15$ & $%
3.23\pm .12$ & $12.29\pm .46$ \\ 
linear & $11$ & $1$ & $1$ & $3.72\pm .16$ & $3.27\pm .14$ & $11.81\pm .35$ & 
$9.18\pm .31$ & $14.40\pm .52$ \\ 
$n=10$ & $51$ & $.5$ & $27$ & $11.19\pm .21$ & $11.05\pm .21$ & $1.24\pm .07$
& $.85\pm .05$ & $4.10\pm .22$ \\ 
& $51$ & $1$ & $27$ & $9.80\pm .23$ & $9.35\pm .22$ & $4.34\pm .22$ & $%
2.96\pm .16$ & $4.82\pm .26$ \\ \hline
& $11$ & $.5$ & $0$ & $5.99\pm 1.25$ & $5.57\pm 1.06$ & $4.62\pm .15$ & $%
3.15\pm .11$ & $14.16\pm .50$ \\ 
quadratic & $11$ & $1$ & $0$ & $7.30\pm 1.82$ & $6.07\pm 1.27$ & $13.04\pm
.38$ & $8.57\pm .38$ & $16.22\pm .57$ \\ 
$n=15$ & $51$ & $.5$ & $4$ & $12.61\pm .46$ & $12.40\pm .45$ & $1.58\pm .07$
& $.99\pm .05$ & $5.75\pm .27$ \\ 
& $51$ & $1$ & $4$ & $10.69\pm .54$ & $10.03\pm .51$ & $5.95\pm .24$ & $%
3.54\pm .17$ & $6.72\pm .31$ \\ \hline
& $11$ & $.5$ & $0$ & $9.71\pm 1.63$ & $9.15\pm 1.52$ & $4.87\pm .15$ & $%
3.25\pm .11$ & $14.90\pm .51$ \\ 
cubic & $11$ & $1$ & $0$ & $12.98\pm 2.53$ & $11.41\pm 2.22$ & $13.54\pm .38$
& $8.86\pm .29$ & $16.92\pm .58$ \\ 
$n=20$ & $51$ & $.5$ & $0$ & $21.67\pm 2.43$ & $21.21\pm 2.38$ & $1.87\pm
.08 $ & $1.14\pm .05$ & $7.24\pm .30$ \\ 
& $51$ & $1$ & $0$ & $20.76\pm 2.9$ & $19.29\pm 2.81$ & $7.39\pm .26$ & $%
3.75\pm .16$ & $8.45\pm .35$ \\ \hline\hline
\end{tabular}%
%TCIMACRO{\TeXButton{E}{\end{table}}}%
%BeginExpansion
\end{table}%
%EndExpansion

%TCIMACRO{\TeXButton{B}{\begin{table}[p] \centering}}%
%BeginExpansion
\begin{table}[tbp] \centering%
%EndExpansion
\begin{tabular}{ccccccccc}
\multicolumn{9}{c}{Table 2. \ Minimax precision matrices; symmetrized
designs:} \\ 
\multicolumn{9}{c}{means of performance measures $\pm $ $1$ standard error.}
\\ \hline\hline
Response & $N$ & $\nu $ & $\%I_{n}$ & $\mathcal{I}_{\nu }\left( \xi
,I_{n}\right) $ & $\mathcal{I}_{\nu }\left( \xi ,P^{\nu }\right) $ & $%
T_{1}\left( \%\right) $ & $T_{2}\left( \%\right) $ & $T_{3}\left( \%\right) $
\\ \hline
& $11$ & $.5$ & $20$ & $2.10\pm .03$ & $2.05\pm .03$ & $2.45\pm .11$ & $%
1.64\pm .07$ & $8.46\pm .41$ \\ 
linear & $11$ & $1$ & $20$ & $1.83\pm .04$ & $1.66\pm .04$ & $8.51\pm .31$ & 
$4.99\pm .21$ & $10.06\pm .46$ \\ 
$n=10$ & $51$ & $.5$ & $80$ & $10.01\pm .29$ & $9.97\pm .29$ & $.40\pm .05$
& $.23\pm .03$ & $1.45\pm .18$ \\ 
& $51$ & $1$ & $80$ & $7.77\pm .29$ & $7.67\pm .29$ & $1.48\pm .16$ & $%
.67\pm .07$ & $1.66\pm .20$ \\ \hline
& $11$ & $.5$ & $0$ & $2.35\pm .08$ & $2.26\pm .07$ & $3.49\pm .07$ & $%
2.17\pm .07$ & $14.10\pm .50$ \\ 
quadratic & $11$ & $1$ & $0$ & $2.01\pm .09$ & $1.74\pm .08$ & $14.40\pm .37$
& $8.57\pm .22$ & $18.05\pm .58$ \\ 
$n=15$ & $51$ & $.5$ & $85$ & $10.58\pm .70$ & $10.53\pm .68$ & $.25\pm .04$
& $.15\pm .02$ & $1.02\pm .16$ \\ 
& $51$ & $1$ & $85$ & $7.54\pm .80$ & $7.39\pm .74$ & $1.03\pm .15$ & $%
.49\pm .07$ & $1.19\pm .19$ \\ \hline
& $11$ & $.5$ & $3$ & $2.64\pm .19$ & $2.55\pm .19$ & $3.56\pm .12$ & $%
1.99\pm .07$ & $15.18\pm .56$ \\ 
cubic & $11$ & $1$ & $3$ & $2.39\pm .27$ & $2.11\pm .26$ & $15.19\pm .44$ & $%
9.09\pm .28$ & $19.69\pm .70$ \\ 
$n=20$ & $51$ & $.5$ & $58$ & $11.97\pm 1.91$ & $11.80\pm 1.80$ & $.45\pm
.04 $ & $.24\pm .02$ & $2.11\pm .19$ \\ 
& $51$ & $1$ & $58$ & $8.44\pm 2.14$ & $7.94\pm 1.80$ & $2.23\pm .18$ & $%
.86\pm .07$ & $2.47\pm .21$ \\ \hline\hline
\end{tabular}%
%TCIMACRO{\TeXButton{E}{\end{table}}}%
%BeginExpansion
\end{table}%
%EndExpansion

In (\ref{max loss}), \textsc{var} $=$ $\mathcal{I}_{0}$ and \textsc{bias} $=$
$\mathcal{I}_{1}$ are the components of the \textsc{imspe} due to variation
and to bias, respectively. That $\mathcal{I}_{0}\left( \xi ,P\right) $ is
minimized by \textsc{ols} for any design was established in \S \ref{sec: gls}%
. If \textsc{ols} is to be a minimax procedure for a particular design and
some $\nu \in (0,1]$, any increase in \textsc{bias} must be outweighed by a
proportional decrease in \textsc{var}. We shall present theoretical and
numerical evidence that for many designs this is not the case, and for these
pairs $\left( \xi ,\nu \right) $ \textsc{ols} is not minimax.

We first present the results of a simulation study, in which the designs
exhibit no particular structure. To find $P$ minimizing (\ref{max loss}) we
use the fact that, by virtue of the Choleski decomposition, any positive
definite matrix can be represented as $P=LL^{\prime }$, for a lower
triangular $L$. We thus express $\mathcal{I}_{\nu }\left( \xi ,LL^{\prime
}\right) $ as a function of the vector $l_{n\left( n+1\right) /2\times 1}$
consisting of the elements in the lower triangle of $L$, and minimize over $%
l $ using a nonlinear constrained minimizer. The constraint -- recall (\ref%
{constraint}) -- is that $l^{\prime }l=n$.

Of course we cannot guarantee that this yields an absolute minimum, but the
numerical evidence is compelling. In any event, the numerical results give a
negative answer to the question of whether or not\ \textsc{ols} is
necessarily minimax -- the minimizing $P$ is often, but not always, the
identity matrix.

In our simulation study we set the design space to be $\chi =\left\{
-1=x_{1}<\cdot \cdot \cdot <x_{N}=1\right\} $, with the $x_{i}$ equally
spaced. We chose regressors $f\left( x\right) =\left( 1,x\right) ^{\prime }$%
, $\left( 1,x,x^{2}\right) ^{\prime }$ or $\left( 1,x,x^{2},x^{3}\right)
^{\prime }$, corresponding to linear, quadratic, or cubic regression. For
various values of $n$ and $N$ we first randomly generated probability
distributions $\left( p_{1},p_{2},..,p_{N}\right) $ and then generated a
multinomial$\left( n;p_{1},p_{2},..,p_{N}\right) $ vector; this is $n\xi $.
For each such design we computed the minimizing $P$, and both components of
the minimized value of $\mathcal{I}_{\nu }\left( \xi ,P\right) $. This was
done for $\nu =.5,1$. We took $n$ equal to five times the number of
regression parameters. Denote by $P^{\nu }$ the minimizing $P$. Of course $%
P^{0}=I_{n}$. In each case we compared three quantities:%
\begin{eqnarray*}
T_{1} &=&100\frac{\left( \mathcal{I}_{\nu }\left( \xi ,P^{0}\right) -%
\mathcal{I}_{\nu }\left( \xi ,P^{\nu }\right) \right) }{\mathcal{I}_{\nu
}\left( \xi ,P^{0}\right) }\text{, the percent reduction in }\mathcal{I}%
_{\nu }\text{ achieved by }P^{\nu }; \\
T_{2} &=&100\frac{\left( \mathcal{I}_{0}\left( \xi ,P^{\nu }\right) -%
\mathcal{I}_{0}\left( \xi ,P^{0}\right) \right) }{\mathcal{I}_{0}\left( \xi
,P^{0}\right) }\text{, the percent increase, relative to \textsc{ols}, in 
\textsc{var};} \\
T_{3} &=&100\frac{\left( \mathcal{I}_{1}\left( \xi ,P^{0}\right) -\mathcal{I}%
_{1}\left( \xi ,P^{\nu }\right) \right) }{\mathcal{I}_{1}\left( \xi
,P^{0}\right) }\text{, the percent decrease, relative to \textsc{ols}, in 
\textsc{bias}.}
\end{eqnarray*}

The means, and standard errors based on 500 runs, of the performance
measures using these `multinomial' designs are given in Table 1.\textbf{\ }%
The percentages of times that $P^{\nu }=I_{n}$ was minimax are also given.
When $\nu =1$ the percent reduction in the bias ($T_{3}$) can be
significant, but is accompanied by an often sizeable increase in the
variance ($T_{2}$). When $\nu =.5$ the reduction $T_{1}$ is typically quite
modest.

These multinomial designs, mimicking those which might arise in
observational studies, are not required to be symmetric. We re-ran the
simulations after symmetrizing the designs by averaging them with their
reflections across $x=0$ and then applying a rounding mechanism which
preserved symmetry. The resulting designs gave substantially reduced losses
both for $P=I_{n}$ (\textsc{ols}) and $P=P^{\nu }$ (\textsc{gls}), and were
much more likely to be optimized by $P^{\nu }=$ $I_{n}$. The differences
between the means of $\mathcal{I}_{\nu }\left( \xi ,I_{n}\right) $ and $%
\mathcal{I}_{\nu }\left( \xi ,P^{\nu }\right) $ were generally statistically
insignificant, and the values of $T_{1}$, $T_{2}$ and $T_{3}$ showed only
very modest benefits to \textsc{gls}. See Table 2.

\begin{remark}From the simulations a user might conclude that even when \textsc{ols} is not minimax, the benefits
of using \textsc{gls} with the minimax $P$ are outweighed by the computational
complexity. An investigator who does decide beforehand to use \textsc{ols}, might then design
to minimize $\mathcal{I}_{\nu }\left( \xi ,P\right) =\mathcal{I}_{\nu
}\left( \xi ,I_{n}\right) $. This is a well-studied problem, solved for
numerous response models under the assumption of i.i.d. errors -- see Wiens
(2015) for a review. We now see that these designs enjoy the additional
property of being minimax against departures $C\in \mathcal{C}_{M}$.
\end{remark}

\subsection{Theoretical complements}

In Theorem \ref{thm: cases} below, we show that the experimenter can often
design in such a way that $P=I_{n}$ is a minimax precision matrix, so that 
\textsc{ols} is a minimax procedure. In particular, this holds if the design
is \textit{uniform} on its support, i.e.\ places an equal number of
observations at each of several points of the design space.

Suppose that a design $\xi $ places $n_{i}\geq 0$ observations at $x_{i}\in
\chi $. Let $J_{+}:n\times q$ be the result of retaining only the non-zero
columns of $J$, so that $JJ^{\prime }=J_{+}J_{+}^{\prime }$, and $%
D_{+}\equiv J_{+}^{\prime }J_{+}\ $is the diagonal matrix containing the
positive $n_{i}$. If the columns removed have labels $j_{1},...,j_{N-q}$
then let $Q_{+}:q\times p$ be the result of removing these rows from $Q$, so
that $JQ=J_{+}Q_{+}$ and $Q^{\prime }DQ=Q_{+}^{\prime }D_{+}Q_{+}$. Now
define $\alpha =n/tr\left( D_{+}^{-1}\right) $ and 
\begin{equation}
P_{0}=\alpha J_{+}D_{+}^{-2}J_{+}^{\prime },  \label{P0}
\end{equation}%
with $trP_{0}=n$. Note that 
\begin{equation*}
rk\left( P_{0}\right) =rk\left( J_{+}D_{+}^{-1}\right) =rk\left(
D_{+}^{-1}J_{+}^{\prime }J_{+}D_{+}^{-1}\right) =rk\left( D_{+}^{-1}\right)
=q,
\end{equation*}%
so that $P_{0}$ is positive definite iff $q=n$. This is relevant in part
(ii) of Theorem \ref{thm: cases}, where we deal with the possible rank
deficiency of $P_{0}$ by introducing 
\begin{equation}
P_{\varepsilon }\equiv \left( P_{0}+\varepsilon I_{n}\right) /\left(
1+\varepsilon \right) ;  \label{Peps}
\end{equation}%
for $\varepsilon >0$, $P_{\varepsilon }$ is positive definite with $tr\left(
P_{\varepsilon }\right) =$ $n$.

\begin{theorem}
\label{thm: cases} (i) Suppose that $q\leq N$ and the design is uniform on $%
q $ points of $\chi $, with $k\geq 1$ observations at each $x_{i}$. Then $%
n=kq$, $D_{+}=kI_{q}$, and $P=I_{n}$ is a minimax precision matrix: 
\begin{equation}
\mathcal{I}_{\nu }\left( \xi ,I_{n}\right) =\min_{P\succ 0}\mathcal{I}_{\nu
}\left( \xi ,P\right) ;  \label{minimaxity}
\end{equation}%
thus \textsc{ols} is minimax within the class of \textsc{gls} methods. In
particular this holds if $P_{0}=$ $I_{n}$, where $P_{0}$ is defined at (\ref%
{P0}). \newline
(ii) \ Suppose that a design $\xi $ places mass on $q\leq N$ points of $\chi 
$, that $P_{0}\neq $ $I_{n}$, and that neither of the following holds: 
%TCIMACRO{\TeXButton{TeX field}{\begin{subequations}}}%
%BeginExpansion
\begin{subequations}%
%EndExpansion
\label{Q-cond} 
\begin{eqnarray}
\left( Q_{+}^{\prime }Q_{+}\right) ^{-1}\left( Q_{+}^{\prime
}D_{+}^{-1}Q_{+}\right) \left( Q_{+}^{\prime }Q_{+}\right) ^{-1} &=&\left(
Q_{+}^{\prime }D_{+}Q_{+}\right) ^{-1},  \label{Q-inverse} \\
ch_{\max }\left\{ \left( Q_{+}^{\prime }D_{+}Q_{+}\right) ^{-1}\left(
Q_{+}^{\prime }D_{+}^{2}Q_{+}\right) \left( Q_{+}^{\prime }D_{+}Q_{+}\right)
^{-1}\right\} &=&ch_{\max }\left\{ \left( Q_{+}^{\prime }Q_{+}\right)
^{-1}\right\} .  \label{Q-square}
\end{eqnarray}%
%TCIMACRO{\TeXButton{TeX field}{\end{subequations}}}%
%BeginExpansion
\end{subequations}%
%EndExpansion
Then in particular $D_{+}$ is not a multiple of $I_{q}$ and so the design is
non-uniform. With $P_{\varepsilon }$ as defined at (\ref{Peps}), there is $%
\nu _{0}\in \left( 0,1\right) $ for which, for each $\nu \in (\nu _{0},1]$, $%
\mathcal{I}_{\nu }\left( \xi ,P_{\varepsilon }\right) <\mathcal{I}_{\nu
}\left( \xi ,I_{n}\right) $. Thus \textsc{ols} is not minimax for such $%
\left( \xi ,\nu \right) $.
\end{theorem}

\begin{remark}
The requirement of Theorem \ref{thm: cases}(ii) that (\ref{Q-inverse}) and (%
\ref{Q-square}) fail excludes more designs than those which are uniform on
their supports, and is a condition on $Q$ as well as on the design. For
instance if $Q_{+}\left( Q_{+}^{\prime }Q_{+}\right) ^{-1/2}\equiv
A_{q\times p}$ is block-diagonal: $A=\oplus _{i=1}^{m}A_{i}$, where $%
A_{i}:q_{i}\times p_{i}$ ($\sum q_{i}=q,\sum p_{i}=p$) satisfies $%
A_{i}^{\prime }A_{i}=I_{p_{i}}$, and if $D_{+}=$ $\oplus
_{i=1}^{m}k_{i}I_{q_{i}}$, then 
\begin{eqnarray}
A^{\prime }D_{+}^{-1}A &=&\left( A^{\prime }D_{+}A\right) ^{-1},
\label{new1} \\
\left( A^{\prime }D_{+}A\right) ^{-1}A^{\prime }D_{+}^{2}A\left( A^{\prime
}D_{+}A\right) ^{-1} &=&I_{p}.  \label{new2}
\end{eqnarray}%
Equation (\ref{new1}) gives (\ref{Q-inverse}), and (\ref{new2}) asserts the
equality of the two matrices in (\ref{Q-square}), hence of their maximum
eigenvalues. These equations are satisfied even though the design is
non-uniform if the $k_{i}$ are not all equal.
\end{remark}

In Tables 1 and 2, uniform designs account for $100\%$ and $95\%$,
respectively, of the cases in which $P^{\nu }=I_{n}$ is optimal. Common
exceptions in Table 2 are designs which are uniform apart from having points
added or removed at $x=0$ to maintain symmetry. Those designs for which $%
I_{n}$ is not optimal all meet the conditions of Theorem \ref{thm: cases}%
(ii). This was checked numerically: since (\ref{Q-inverse})\ implies that $%
\mathcal{I}_{0}\left( \xi ,P_{0}\right) -\mathcal{I}_{0}\left( \xi
,I_{n}\right) =0$, and (\ref{Q-square}) \ implies that $\mathcal{I}%
_{1}\left( \xi ,I_{n}\right) -\mathcal{I}_{1}\left( \xi ,P_{0}\right) =0$,
their failure is verified by checking that each of these differences is
positive.

\section{Minimax precision matrices and minimax designs\label{sec:
combination}}

%TCIMACRO{\TeXButton{B}{\begin{table}[p] \centering}}%
%BeginExpansion
\begin{table}[tbp] \centering%
%EndExpansion
\begin{tabular}{ccccccccc}
\multicolumn{9}{c}{Table 3. \ Minimax designs and precision matrices:} \\ 
\multicolumn{9}{c}{performance measures ($T_{1}=T_{2}=T_{3}=0$ if $P^{\nu
}=I_{n}$).} \\ \hline\hline
Response & $N$ & $\nu $ &  & $\mathcal{I}_{\nu }\left( \xi ,I_{n}\right) $ & 
$\mathcal{I}_{\nu }\left( \xi ,P^{\nu }\right) $ & $T_{1}(\%)$ & $T_{2}(\%)$
& $T_{3}(\%)$ \\ \hline
& $11$ & $.5$ &  & $1.60$ & $1.60$ & $0$ & $0$ & $0$ \\ 
linear & $11$ & $1$ &  & $1.10$ & $1.10$ & $0$ & $0$ & $0$ \\ 
$n=10$ & $51$ & $.5$ &  & $6.14$ & $6.14$ & $0$ & $0$ & $0$ \\ 
& $51$ & $1$ &  & $5.10$ & $5.10$ & $0$ & $0$ & $0$ \\ \hline
& $11$ & $.5$ &  & $1.61$ & $1.53$ & $4.67$ & $2.81$ & $16.06$ \\ 
quadratic & $11$ & $1$ &  & $1.12$ & $1.00$ & $10.84$ & $10.78$ & $10.84$ \\ 
$n=15$ & $51$ & $.5$ &  & $5.80$ & $5.80$ & $0$ & $0$ & $0$ \\ 
& $51$ & $1$ &  & $3.40$ & $3.40$ & $0$ & $0$ & $0$ \\ \hline
& $11$ & $.5$ &  & $1.55$ & $1.53$ & $1.46$ & $1.13$ & $6.04$ \\ 
cubic & $11$ & $1$ &  & $1.12$ & $1.00$ & $11.03$ & $4.84$ & $11.03$ \\ 
$n=20$ & $51$ & $.5$ &  & $5.56$ & $5.56$ & $0$ & $0$ & $0$ \\ 
& $51$ & $1$ &  & $2.55$ & $2.55$ & $0$ & $0$ & $0$ \\ \hline\hline
\end{tabular}%
%TCIMACRO{\TeXButton{E}{\end{table}}}%
%BeginExpansion
\end{table}%
%EndExpansion
\begin{figure}[btp]
\centering
\includegraphics[scale=1]{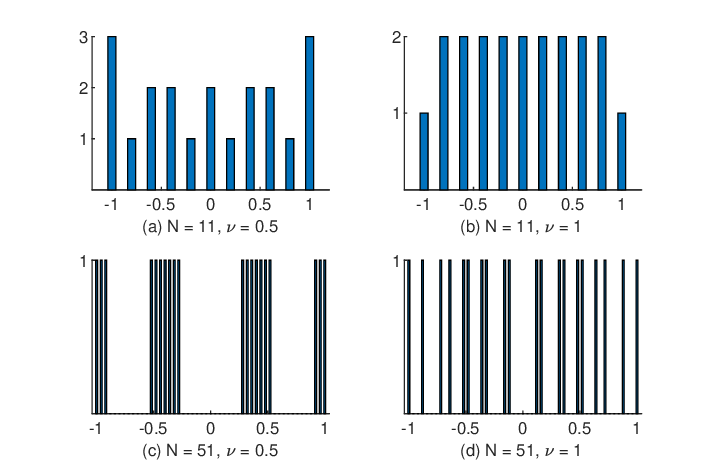}
\caption{Minimax design frequencies for a
cubic model; $n=20$.}
\label{fig: designs}
\end{figure}

We investigated the interplay between minimax precision matrices and minimax
designs. To this end (\ref{max loss}) was minimized over both $\xi $ and $P$%
. To minimize over $\xi $ we employed particle swarm optimization (\cite{ke95}). 
The algorithm searches over continuous designs $\xi $,
and so each such design to be evaluated was first rounded so that $n\xi $
had integer values. Then $J$, and the corresponding minimax precision matrix 
$P^{\nu }=P^{\nu }\left( J\right) $ were computed and the loss returned. The
final output is an optimal pair $\left\{ J^{\nu },P^{\nu }\right\} $. Using
a genetic algorithm yielded the same results but was many times slower.

The results, using the same parameters as in Tables 1 and 2, are shown in
Table 3. We note that in all cases the use of the minimax design gives
significantly smaller losses, both using \textsc{ols} and \textsc{gls}. In
eight of the twelve cases studied it turns out that the minimax design is
uniform on its support and so the choice $P^{\nu }=I_{n}$ is minimax. In the
remaining cases -- all in line with (iii) of Theorem \ref{thm: cases} --
minimax precision results in only a marginal improvement. Of the two factors
-- $\xi $ and $P$ -- explaining the decrease in $\mathcal{I}_{\nu }$, the
design is by far the greater contributor.

See Figure \ref{fig: designs} for some representative plots of the minimax
designs for a cubic response. These reflect several common features of
robust designs. One is that the designs using $\nu =1$, i.e.\ aimed at
minimization of the bias alone, tend to be more uniform than those using $%
\nu =.5$. This reflects the fact -- following from (\ref{orthogonality}) and
exploited in (i) of Theorem \ref{thm: cases} -- that when a uniform design
on all of $\chi $ is implemented, then the bias using \textsc{ols} vanishes.
As well, when the design space is sufficiently rich as to allow for clusters
of nearby design points to replace replicates, then this invariably occurs.
See \cite{w24b}, \cite{fw00} and \cite{hsw01} for examples
and discussions. Such clusters form near the support points of the
classically I-optimal designs, minimizing variance alone. See for instance
\cite{s77} who showed that the I-optimal design for cubic regression
places masses $.1545$, $.3455$ at each of $\pm 1$, $\pm .477$ - a situation
well approximated by the design in (c) of Figure \ref{fig: designs}, whose
clusters around these points account for masses of $.15$ and $.35$ each. As $%
\nu $ increases the clusters become more diffuse, as in d) of Figure \ref%
{fig: designs}. A result of our findings in this article is that an
additional benefit to such clustering is that it allows \textsc{ols} to be a
minimax \textsc{gls} procedure.

\section{Continuous design spaces}

If the design space is continuous -- an interval, or hypercube, for instance
-- then the problem of finding a minimax design is somewhat more
complicated. We continue to work with the alternate models (\ref{alternate
response}), but the constraints (\ref{orthogonality}) and (\ref{tau-bound})
now have their finite sums replaced by Lebesgue integrals over $\chi $. As
shown in \cite{w92}, in order that a design $\xi $ have finite maximum
loss it is necessary that it be absolutely continuous, i.e. possess a
density. Otherwise it places positive mass on some set of Lebesgue measure
zero, on which any $\psi \left( x\right) $ can be modified without affecting
these integrals. Such modifications can be done in such a way as to make the
integrated squared bias unbounded. Thus static, discrete designs are not admissible.

A remedy, detailed by \cite{ww22}, is to choose design points randomly, from an appropriate density. In
the parlance of game theory, this precludes Nature, assumed malevolent, from
knowing the support points of $\xi $ and replying with a $\psi \left( \cdot
\right) $ modified as above. They also recommend designs concentrated near
the I-optimal design points, leading to (c) and (d) of Figure \ref{fig:
designs}, but with randomly chosen clusters. The resulting designs are
always uniform on their supports, and so \textsc{ols} is minimax in all
cases, in contrast to the situation illustrated in (a) and (b) of Figure \ref%
{fig: designs}. See \cite{w24b} for guidelines and further examples, and \cite{wt24} for extensions allowing for randomized replication. For results on robustness of inferences, see \cite{zz}.

\appendix

\section*{Appendix: Derivations}

\setcounter{equation}{0} \renewcommand{\theequation}{A.\arabic{equation}}

\noindent \textit{Details for Remark 1:}
We take $\Phi \left( \Sigma \right) =\det\left( \Sigma \right)$ and consider the ratio $r(X,C_{0})$ of the maximum loss of the $\textsc{gls}$ estimate to that of the $\textsc{ols}$ estimate. This is
\begin{equation*} 
r(X,C_{0}) = \frac{\max_{\mathcal{C}_{M}}\mathcal{L}\left( C\left\vert {}\right. P_{0}\right)}  {\max_{\mathcal{C}_{M}}\mathcal{L}\left( C\left\vert {}\right. I_{n}\right)}
=\frac{\mathcal{L}\left( \eta^{2} I_{n}  \left\vert {}\right. P_{0}\right)}
{\mathcal{L}\left( \eta^{2} I_{n} \left\vert {}\right. I_{n}\right)}
= \frac{| X^{\prime }C_{0}^{-2}X||X^{\prime}X|}{| X^{\prime }C_{0}^{-1}X|^{2}} 
= \frac{| Q^{\prime }C_{0}^{-2}Q|}{| Q^{\prime }C_{0}^{-1}Q|^{2}},
\end{equation*}
where the final term employs the QR-decomposition of $X$. Of course $r(X,C_{0}) \ge 1$; a direct proof follows from the observation that
\begin{equation*}
Q^{\prime }C_{0}^{-2}Q - \left(Q^{\prime }C_{0}^{-1}Q\right)^{2} =
Q^{\prime }C_{0}^{-1}\left\{I_{n} - Q Q^{\prime}\right\} C_{0}^{-1}Q
\end{equation*}
is p.s.d since the matrix in braces is idempotent, hence p.s.d.

For the equicorrelation model with $\rho >0$ and $C_{0}=\left( 1-\rho \right) \left(
I_{n}+\alpha 1_{n}1_{n}^{\prime }\right) $ for $\alpha =\rho
/\left( 1-\rho \right) \in \left( 0,\infty \right) $, we calculate that%
\begin{eqnarray*}
C_{0}^{-1} &=&\left( 1-\rho \right) ^{-1}\left( I_{n}-\beta 1_{n}1_{n}^{\prime
}\right) \text{ for }\beta =\alpha /\left( 1+n\alpha \right) \in \left(
0,n^{-1}\right), \\
C_{0}^{-2} &=&\left( 1-\rho \right) ^{-2}\left( I_{n}-\gamma 1_{n}1_{n}^{T
}\right) \text{ for }\gamma =2\beta - n\beta ^{2}\in \left(0,n^{-1}\right),
\end{eqnarray*}%
and then   
\begin{equation*}
r\left( X,C_{0}\right) =1+\frac{S\beta ^{2}\left( n-S\right) }{\left(
1-S\beta \right) ^{2}},
\end{equation*}%
for $S=1_{n}^{\prime }QQ^{\prime }1_{n}\in \left[ 0,n\right] $. 

For $0<\varepsilon <n^{-1}$ suppose that $\rho $ is sufficiently large that $\beta
=n^{-1}-\varepsilon $, and that $S=n-\varepsilon $. Then 
\begin{equation*}
r\left( X,C_{0}\right) =1+\frac{\phi _{n}\left( \varepsilon \right) }{%
\varepsilon } \text{ for }\phi _{n}\left( \varepsilon \right) =\frac{\left(
n-\varepsilon \right) \left( n^{-1}-\varepsilon \right) ^{2}}{\left(
n+n^{-1}-\varepsilon \right) ^{2}}>0\text{ },
\end{equation*}%
and $r\left( X,C_{0}\right) \rightarrow \infty $ as $\varepsilon \rightarrow
0$ and $\phi _{n}\left( \varepsilon \right)  \rightarrow 1/(1+n^{2})$.

A simple example in which $S=n-\varepsilon $ is attained has $p=1$, $X=x_{n \times 1}$. Then $Q=x/\left\Vert
x\right\Vert $ and, with $\sigma _{X}^{2}=\left( \sum x^{2}-n\overline{x%
}^{2}\right) /n$,  we have $S = n\overline{x}^{2}/(\sigma _{X}^{2}+\overline{x}^{2})$. Then $S = n-\varepsilon$ if the $x_{i}$ are sufficiently concentrated that $\sigma _{X}^{2}=\varepsilon \overline{x}^{2}/(n-\varepsilon )$. This example extends easily to arbitrary $p$.

\begin{proof}[Proof of Theorem \protect\ref{thm: maxima}]
In the notation\ of the theorem, (\ref{I-cond0}) becomes 
\begin{eqnarray}
\mathcal{I}\left( \psi ,\xi \left\vert {}\right. C,P\right) &=&tr\left\{ F%
\text{\textsc{cov}}\left( \hat{\theta}\left\vert {}\right. C,P\right)
F^{\prime }\right\}  \notag \\
&+&\bar{\psi}^{\prime }J^{\prime }PJF\left( F^{\prime }J^{\prime }PJF\right)
^{-1}F^{\prime }F\left( F^{\prime }J^{\prime }PJF\right) ^{-1}F^{\prime
}J^{\prime }PJ\bar{\psi}+\bar{\psi}^{\prime }\bar{\psi}.  \label{I-loss-0}
\end{eqnarray}%
As in \S \ref{sec: gls}, and taking $K=F^{\prime }F$ in (iv) of that
section, for $C\in \mathcal{C}_{M}$ the trace in (\ref{I-loss-0}) is
maximized by $C=\eta ^{2}I_{n}$, with%
\begin{equation}
trF\text{\textsc{cov}}\left( \hat{\theta}\left\vert {}\right. \eta
^{2}I_{n},P\right) F^{\prime }=\eta ^{2}tr\left\{ F\left( F^{\prime
}J^{\prime }PJF\right) ^{-1}\left( F^{\prime }J^{\prime }P^{2}JF\right)
\left( F^{\prime }J^{\prime }PJF\right) ^{-1}F^{\prime }\right\} .
\label{max cov}
\end{equation}%
Extend the orthogonal basis for $\func{col}\left( F\right) $ -- formed by
the columns of $Q$ \ -- by appending to $Q$ the matrix $Q_{\ast }:N\times
\left( N-p\right) $, whose columns form an orthogonal basis for the
orthogonal complement $\func{col}\left( F\right) ^{\perp }$. Then $(Q\vdots
Q_{\ast }):N\times N$ is an orthogonal matrix and we have that$\ F=$ $QR$
for a non-singular $R$. If the construction is carried out by the
Gram-Schmidt method, then $R$ is upper triangular. \newline
Constraint (\ref{orthogonality}) dictates that $\bar{\psi}$ lie in $\func{col}%
\left( Q_{\ast }\right) $. A maximizing $\psi $ will satisfy (\ref{tau-bound}%
) with equality, hence $\bar{\psi}=\tau Q_{\ast }\beta $ for some $\beta
_{\left( N-p\right) \times 1}$ with unit norm. Combining these observations
along with (\ref{I-loss-0}) and (\ref{max cov}) yields that $\max_{\psi ,C}%
\mathcal{I}\left( \psi ,\xi \left\vert {}\right. C,P\right) $ is given by 
\begin{eqnarray}
&&\eta ^{2}tr\left\{ Q\left( Q^{\prime }UQ\right) ^{-1}\left( Q^{\prime
}VQ\right) \left( Q^{\prime }UQ\right) ^{-1}Q^{\prime }\right\}  \notag \\
&&+\tau ^{2}\max_{\left\Vert \beta \right\Vert =1}\left\{ \beta ^{\prime
}Q_{\ast }^{\prime }UQ\left( Q^{\prime }UQ\right) ^{-1}Q^{\prime }Q\left(
Q^{\prime }UQ\right) ^{-1}Q^{\prime }UQ_{\ast }\beta +1\right\} .
\label{I-loss-1}
\end{eqnarray}%
Here and elsewhere we use that $trAB=trBA$, and that such products have the
same non-zero eigenvalues. Then (\ref{I-loss-1}) becomes $\left( \tau
^{2}+\eta ^{2}\right) $ times $\mathcal{I}_{\nu }\left( \xi ,P\right) $,
given by%
\begin{eqnarray}
&&\mathcal{I}_{\nu }\left( \xi ,P\right) =\left( 1-\nu \right) tr\left\{
\left( Q^{\prime }UQ\right) ^{-1}\left( Q^{\prime }VQ\right) \left(
Q^{\prime }UQ\right) ^{-1}\right\}  \notag \\
&&+\nu \left\{ ch_{\max }Q_{\ast }^{\prime }UQ\left( Q^{\prime }UQ\right)
^{-1}\cdot \left( Q^{\prime }UQ\right) ^{-1}Q^{\prime }UQ_{\ast }+1\right\} .
\label{I-loss-2}
\end{eqnarray}%
The maximum eigenvalue is also that of 
\begin{eqnarray*}
\left( Q^{\prime }UQ\right) ^{-1}Q^{\prime }UQ_{\ast }\cdot Q_{\ast
}^{\prime }UQ\left( Q^{\prime }UQ\right) ^{-1} &=&\left( Q^{\prime
}UQ\right) ^{-1}Q^{\prime }U\left( I_{N}-QQ^{\prime }\right) UQ\left(
Q^{\prime }UQ\right) ^{-1} \\
&=&\left( Q^{\prime }UQ\right) ^{-1}Q^{\prime }U^{2}Q\left( Q^{\prime
}UQ\right) ^{-1}-I_{p};
\end{eqnarray*}%
this in (\ref{I-loss-2}) gives (\ref{max loss}).
\end{proof}

The proof of Theorem \ref{thm: cases} requires a preliminary result.

\begin{lemma}
\label{thm: lemma} (i) For a fixed design $\xi $ and any $P\succ 0$, $%
\mathcal{I}_{0}\left( \xi ,P\right) \geq \mathcal{I}_{0}\left( \xi
,I_{n}\right) $ and $\mathcal{I}_{1}\left( \xi ,P\right) \geq ch_{\max
}\left\{ \left( Q_{+}^{\prime }Q_{+}\right) ^{-1}\right\} $. \newline
If neither of the equations (\ref{Q-inverse}), (\ref{Q-square}) holds, then: 
\newline
(ii) $\mathcal{I}_{0}\left( \xi ,P_{0}\right) >\mathcal{I}_{0}\left( \xi
,I_{n}\right) $ \ and $\mathcal{I}_{1}\left( \xi ,I_{n}\right) >\mathcal{I}%
_{1}\left( \xi ,P_{0}\right) $; \newline
(iii) With $P_{\varepsilon }$ as defined in Theorem \ref{thm: cases}(ii),
and for sufficiently small $\varepsilon >$ $0$, $\Delta _{0}\left(
\varepsilon \right) =\mathcal{I}_{0}\left( \xi ,P_{\varepsilon }\right) -%
\mathcal{I}_{0}\left( \xi ,I_{n}\right) >0$ and $\Delta _{1}\left(
\varepsilon \right) =\mathcal{I}_{1}\left( \xi ,I_{n}\right) -\mathcal{I}%
_{1}\left( \xi ,P_{\varepsilon }\right) >$ $0$.
\end{lemma}

\begin{proof}[Proof of Lemma \protect\ref{thm: lemma}]
(i) From (\ref{max var}), $\mathcal{I}_{0}\left( \xi ,P\right) -\mathcal{I}%
_{0}\left( \xi ,I_{n}\right) $ is the trace of%
\begin{eqnarray*}
&&\left( Q^{\prime }J^{\prime }PJQ\right) ^{-1}\left( Q^{\prime }J^{\prime
}P^{2}JQ\right) \left( Q^{\prime }J^{\prime }PJQ\right) ^{-1}-\left(
Q^{\prime }J^{\prime }JQ\right) ^{-1} \\
&=&\left( Q^{\prime }J^{\prime }PJQ\right) ^{-1}Q^{\prime }J^{\prime
}P\left\{ I_{n}-JQ\left( Q^{\prime }J^{\prime }JQ\right) ^{-1}Q^{\prime
}J^{\prime }\right\} PJQ\left( Q^{\prime }J^{\prime }PJQ\right) ^{-1},
\end{eqnarray*}%
which is $\succeq 0$ (since the matrix in braces is idempotent, hence
p.s.d.) with non-negative trace. For the second inequality first note that 
\begin{eqnarray*}
&&\left( Q^{\prime }UQ\right) ^{-1}Q^{\prime }U^{2}Q\left( Q^{\prime
}UQ\right) ^{-1}-\left( Q_{+}^{\prime }Q_{+}\right) ^{-1} \\
&=&\left( Q_{+}^{\prime }J_{+}^{\prime }PJ_{+}Q_{+}\right)
^{-1}Q_{+}^{\prime }J_{+}^{\prime }PJ_{+}J_{+}^{\prime }PJ_{+}Q_{+}\left(
Q_{+}^{\prime }J_{+}^{\prime }PJ_{+}Q_{+}\right) ^{-1}-\left( Q_{+}^{\prime
}Q_{+}\right) ^{-1} \\
&=&\left( Q_{+}^{\prime }J_{+}^{\prime }PJ_{+}Q_{+}\right)
^{-1}Q_{+}^{\prime }J_{+}^{\prime }PJ_{+}\left\{ I_{n}-Q_{+}\left(
Q_{+}^{\prime }Q_{+}\right) ^{-1}Q_{+}^{\prime }\right\} J_{+}^{\prime
}PJ_{+}Q_{+}\left( Q_{+}^{\prime }J_{+}^{\prime }PJ_{+}Q_{+}\right) ^{-1}
\end{eqnarray*}%
is p.s.d., so that by Weyl's Monotonicity Theorem (\cite{b97}),%
\begin{equation*}
\mathcal{I}_{1}\left( \xi ,P\right) =ch_{\max }\left\{ \left( Q^{\prime
}UQ\right) ^{-1}Q^{\prime }U^{2}Q\left( Q^{\prime }UQ\right) ^{-1}\right\}
\geq ch_{\max }\left\{ \left( Q_{+}^{\prime }Q_{+}\right) ^{-1}\right\} .%
\newline
\end{equation*}%
(ii) We use the following identities, which follow from (\ref{max var}) and (%
\ref{max bias}), expressed in the notation preceding the statement of
Theorem \ref{thm: cases}: 
%TCIMACRO{\TeXButton{TeX field}{\begin{subequations}}}%
%BeginExpansion
\begin{subequations}%
%EndExpansion
\label{P-losses} 
\begin{eqnarray}
\mathcal{I}_{0}\left( \xi ,I_{n}\right)  &=&tr\left\{ \left( Q_{+}^{\prime
}D_{+}Q_{+}\right) ^{-1}\right\}   \label{I-var} \\
\mathcal{I}_{1}\left( \xi ,I_{n}\right)  &=&ch_{\max }\left\{ \left(
Q_{+}^{\prime }D_{+}Q_{+}\right) ^{-1}\left( Q_{+}^{\prime
}D_{+}^{2}Q_{+}\right) \left( Q_{+}^{\prime }D_{+}Q_{+}\right) ^{-1}\right\} 
\label{I-bias} \\
\mathcal{I}_{0}\left( \xi ,P_{0}\right)  &=&tr\left\{ \left( Q_{+}^{\prime
}Q_{+}\right) ^{-1}\left( Q_{+}^{\prime }D_{+}^{-1}Q_{+}\right) \left(
Q_{+}^{\prime }Q_{+}\right) ^{-1}\right\} ,  \label{P-var} \\
\mathcal{I}_{1}\left( \xi ,P_{0}\right)  &=&ch_{\max }\left\{ \left(
Q_{+}^{\prime }Q_{+}\right) ^{-1}\right\} .\newline
\label{P-bias}
\end{eqnarray}%
%TCIMACRO{\TeXButton{TeX field}{\end{subequations}}}%
%BeginExpansion
\end{subequations}%
%EndExpansion
To prove (ii) we show that if either inequality fails then one of (\ref%
{Q-inverse}), (\ref{Q-square}) holds -- a contradiction. First note that 
\begin{eqnarray}
&&\mathcal{I}_{0}\left( \xi ,P_{0}\right) -\mathcal{I}_{0}\left( \xi
,I_{n}\right)   \notag \\
&=&tr\left\{ \left( Q_{+}^{\prime }Q_{+}\right) ^{-1}\left( Q_{+}^{\prime
}D_{+}^{-1}Q_{+}\right) \left( Q_{+}^{\prime }Q_{+}\right) ^{-1}-\left(
Q_{+}^{\prime }D_{+}Q_{+}\right) ^{-1}\right\}   \label{diff trace} \\
&=&tr\left\{ \left( Q_{+}^{\prime }Q_{+}\right) ^{-1}Q_{+}^{\prime
}D_{+}^{-1/2}\left[ I_{q}-D_{+}^{1/2}Q_{+}\left( Q_{+}^{\prime
}D_{+}Q_{+}\right) ^{-1}Q_{+}^{\prime }D_{+}^{1/2}\right] D_{+}^{-1/2}Q_{+}%
\left( Q_{+}^{\prime }Q_{+}\right) ^{-1}\right\} ,  \notag
\end{eqnarray}%
which is non-negative. If the first inequality fails, so that $\mathcal{I}%
_{0}\left( \xi ,P_{0}\right) =\mathcal{I}_{0}\left( \xi ,I_{n}\right) $,
then the trace of the p.s.d. matrix at (\ref{diff trace}) is zero, hence all
eigenvalues are zero and the matrix is the zero matrix. This is (\ref%
{Q-inverse}). \newline
That $\mathcal{I}_{1}\left( \xi ,I_{n}\right) -\mathcal{I}_{1}\left( \xi
,P_{0}\right) \geq 0$ is the first inequality in (i). If the second
inequality of (ii) fails, then $\mathcal{I}_{1}\left( \xi ,I_{n}\right) =%
\mathcal{I}_{1}\left( \xi ,P_{0}\right) $ and their evaluations at (\ref%
{I-bias}) and (\ref{P-bias}) give (\ref{Q-square}). \newline
For (iii), that $\Delta _{0}\left( \varepsilon \right) >0$ and $\Delta
_{1}\left( \varepsilon \right) >0$ for sufficiently small $\varepsilon $
follow from the continuity of $\mathcal{I}_{0}\left( \xi ,P_{\varepsilon
}\right) $ and $\mathcal{I}_{1}\left( \xi ,P_{\varepsilon }\right) $ as
functions of $\varepsilon $: $\Delta _{0}\left( \varepsilon \right)
\rightarrow \mathcal{I}_{0}\left( \xi ,P_{0}\right) -\mathcal{I}_{0}\left(
\xi ,I_{n}\right) >0$ and $\Delta _{1}\left( \varepsilon \right) =\mathcal{I}%
_{1}\left( \xi ,I_{n}\right) -\mathcal{I}_{1}\left( \xi ,P_{0}\right) >0$ as 
$\varepsilon \rightarrow 0$.
\end{proof}

\medskip 

\begin{proof}[Proof of Theorem \protect\ref{thm: cases}]
(i) From the first inequality in Lemma \ref{thm: lemma} (i), 
\begin{equation*}
\mathcal{I}_{0}\left( \xi ,I_{n}\right) =\min_{P\succ 0}\mathcal{I}%
_{0}\left( \xi ,P\right) .
\end{equation*}%
If $P=I_{n}$ then $U=J^{\prime }PJ=J^{\prime }J=D_{+}=kI_{q}$, so that from (%
\ref{I-bias}), and the second inequality in Lemma \ref{thm: lemma}(i), 
\begin{equation*}
\mathcal{I}_{1}\left( \xi ,I_{n}\right) =ch_{\max }\left\{ \left(
Q_{+}^{\prime }Q_{+}\right) ^{-1}\right\} =\min_{P\succ 0}\mathcal{I}%
_{1}\left( \xi ,P\right) .
\end{equation*}%
Now (\ref{minimaxity}) is immediate. If $P_{0}=I_{n}$ then $q=rk\left(
P_{0}\right) =n$, so that all $n$ observations are made at distinct points,
hence $D_{+}=I_{n}$ and the design is uniform on its support. $\newline
$(ii) \ By Lemma \ref{thm: lemma}(iii) there is $\varepsilon _{0}>0$ for
which $\Delta _{0}\left( \varepsilon \right) >0$ and $\Delta _{1}\left(
\varepsilon \right) >$ $0$ when $0<\varepsilon \leq \varepsilon _{0}$. For $%
\varepsilon $ in this range, 
\begin{equation*}
\mathcal{I}_{\nu }\left( \xi ,I_{n}\right) -\mathcal{I}_{\nu }\left( \xi
,P_{\varepsilon }\right) =\nu \left( \Delta _{0}\left( \varepsilon \right)
+\Delta _{1}\left( \varepsilon \right) \right) -\Delta _{0}\left(
\varepsilon \right) >0,
\end{equation*}%
for $\nu \in (\nu _{0},1]$ and $\nu _{0}\equiv \Delta _{0}\left( \varepsilon
\right) /\left( \Delta _{0}\left( \varepsilon \right) +\Delta _{1}\left(
\varepsilon \right) \right) $.
\end{proof}

\section*{Acknowledgements}

This work was carried out with the support of the Natural Sciences and
Engineering Research Council of Canada.

\bibliography{references}
\end{document}